\documentclass[11pt]{amsart}
\usepackage[margin=30mm]{geometry}
\usepackage{amsmath,amssymb}
\usepackage{amsthm}
\usepackage{mathrsfs}

\newtheorem{thm}{Theorem}

\newtheorem{lem}{Lemma}
\newtheorem{cor}{Corollary}
\newtheorem{defi}{Definition}

\newtheorem{rem}{Remark}

\newtheorem{theorem}{Theorem}[section]
\newtheorem{lemma}{Lemma}[section]

\usepackage{enumitem}

\begin{document}

\title{A note on Bohr chaos and hyperbolic sets}
\author{Noriaki Kawaguchi}
\subjclass[2020]{37B05, 37B65, 37D05, 37D45}
\keywords{shadowing, expansiveness, Bohr chaos, hyperbolic sets, chain components}
\address{Department of Mathematical and Computing Science, School of Computing, Institute of Science Tokyo, 2-12-1 Ookayama, Meguro-ku, Tokyo 152-8552, Japan}
\email{gknoriaki@gmail.com}

\begin{abstract}
This paper studies the relationship between shadowing phenomena and Bohr chaos in dynamical systems. We provide sufficient conditions for Bohr chaos in terms of shadowing. By combining those conditions with the shadowing lemma, we obtain some results on Bohr chaos and hyperbolic sets. Our results also highlight some simple but non-trivial structural properties of hyperbolic sets.
\end{abstract}

\maketitle

\markboth{NORIAKI KAWAGUCHI}{A note on Bohr chaos and hyperbolic sets}

{\em Shadowing} is an important notion in the qualitative theory of dynamical systems (see \cite{AH,P} for background). It describes the phenomenon in which coarse orbits, or {\em pseudo-orbits}, can be closely approximated by true orbits. The classical shadowing lemma asserts that a $C^1$-diffeomorphism of a closed differentiable manifold exhibits shadowing on a neighborhood of its hyperbolic set.  As a feature of hyperbolic dynamics, shadowing is intimately related to the intricate behaviors collectively called as `chaos'.

{\em Chaos} has long been a central subject in the theory of dynamical systems and numerous definitions of chaos have been proposed to capture various aspects of complex dynamical behavior. {\em Bohr chaos} is the (relatively recent) notion introduced in \cite{FFRS,FSV}. In this paper, we provide sufficient conditions for Bohr chaos in terms of shadowing. Moreover, by combining those conditions with the shadowing lemma, we obtain some results concerning Bohr chaos and hyperbolic sets. Our results also highlight some simple but non-trivial structural properties of hyperbolic sets that play a role in the emergence of chaotic behavior.

We begin by recalling the definition of Bohr chaos from \cite{FFRS,FSV}. Throughout, $X$ denotes a compact metric space endowed with a metric $d\colon X\times X\to[0,\infty)$. We denote by $C(X)$ the set of continuous real-valued functions on $X$. A sequence $(a_i)_{i\ge0}$ of real numbers is said to be {\em bounded} if
\[
\sup_{i\ge0}|a_i|<\infty.
\]
For a continuous map $f\colon X\to X$, we say that a bounded sequence $(a_i)_{i\ge0}$ of real numbers is {\em orthogonal} to $(X,f)$ if
\[
\lim_{n\to\infty}\frac{1}{n}\sum_{i=0}^{n-1}a_i\phi(f^i(x))=0
\]
for all $x\in X$ and all $\phi\in C(X)$. A continuous map $f\colon X\to X$ is said to be {\em Bohr chaotic} if every bounded sequence $(a_i)_{i\ge0}$ of real numbers with
\[
\limsup_{n\to\infty}\frac{1}{n}\sum_{i=0}^{n-1}|a_i|>0
\]
is not orthogonal to $(X,f)$, i.e., satisfies
\[
\limsup_{n\to\infty}\frac{1}{n}\left|\sum_{i=0}^{n-1}a_i\phi(f^i(x))\right|>0.
\]
for some $x\in X$ and some $\phi\in C(X)$.

\begin{rem}
\normalfont
Let $f\colon X\to X$ be a continuous map and let $h_{\rm top}(f)$ denote the topological entropy of $f$. 
\begin{itemize}
\item For a continuous self-map $g\colon Y\to Y$ of a compact metric space $Y$,
\begin{itemize}
\item if $f$ is Bohr chaotic and there is a injective continuous map $h\colon X\to Y$ such that $h\circ f=g\circ h$, then $g$ is Bohr chaotic,
\item if $g$ is Bohr chaotic and there is a surjective continuous map $h\colon X\to Y$  such that $h\circ f=g\circ h$, then $f$ is Bohr chaotic.
\end{itemize}
\item If $f$ is Bohr chaotic, then $h_{\rm top}(f)>0$ (see Introduction of \cite{FFRS}).
\item If $f$ is uniquely ergodic, then $f$ is not Bohr chaotic; therefore, $h_{\rm top}(f)>0$ does not necessarily imply that $f$ is Bohr chaotic (see \cite[Theorem 1.3]{FFRS}).
\item If there are $m\ge1$ and a closed subset $Y$ of $X$ with $f^m(Y)\subset Y$ for which there is a homeomorphism $h\colon Y\to\{0,1\}^\mathbb{N}$ (or $h\colon Y\to\{0,1\}^\mathbb{Z}$) such that
\[
h\circ f^m|_Y=\sigma\circ h,
\]
where $\sigma$ is the shift map, then $f$ is Bohr chaotic (see \cite[Theorem 1.1]{FFRS}). In \cite[Question 1]{FFRS}, it is asked if the injectivity assumption of $h$ can be removed. This question is resolved in \cite{HLT} (see \cite[Theorem A]{HLT}).
\item If $f$ admits the specification, then $f$ is Bohr chaotic (see \cite[Theorem 4.2.1]{T}).
\end{itemize}
\end{rem}

We shall present simple sufficient conditions for Bohr chaos in terms of shadowing. Let us first introduce some notations and definitions.

Given a homeomorphism $f\colon X\to X$ and $x\in X$, we denote by $\mathcal{O}_f(x)$ the orbit of $x$ for $f$:
\[
\mathcal{O}_f(x)=\{f^i(x)\colon i\in\mathbb{Z}\}.
\]
For a closed subset $S$ of $X$ with $f(S)=S$, we define subsets $W^u(S),W^s(S)$ of $X$ by
\[
W^u(S)=\{x\in X\colon\lim_{i\to-\infty}d(f^i(x),S)=0\},
\]
\[
W^s(S)=\{x\in X\colon\lim_{i\to+\infty}d(f^i(x),S)=0\}.
\]

The definition of shadowing on a subset is as below.

\begin{defi}
\normalfont
Let $f\colon X\to X$ be a homeomorphism and let $\xi=(x_i)_{i\in\mathbb{Z}}$ be a sequence of points in $X$. For $\delta>0$, $\xi$ is said to be a {\em $\delta$-pseudo orbit} of $f$ if
\[
\sup_{i\in\mathbb{Z}}d(f(x_i),x_{i+1})\le\delta.
\]
For $x\in X$ and $\epsilon>0$, $\xi$ is said to be {\em $\epsilon$-shadowed} by $x$ if
\[
\sup_{i\in\mathbb{Z}}d(x_i,f^i(x))\le\epsilon.
\]
For a subset $S$ of $X$, we say that $f$ has {\em shadowing} on $S$ if for any $\epsilon>0$, there is $\delta>0$ such that every $\delta$-pseudo orbit $(x_i)_{i\in\mathbb{Z}}$ of $f$ with $x_i\in S$ for all $i\in\mathbb{Z}$ is $\epsilon$-shadowed by some $x\in X$. 
\end{defi}

Given a homeomorphism $f\colon X\to X$ and $\delta>0$, a finite sequence $(x_i)_{i=0}^k$ of points in $X$, where $k\ge1$, is said to be a {\em $\delta$-chain} of $f$ if
\[
\sup_{0\le i\le k-1}d(f(x_i),x_{i+1})\le\delta.
\]
We say that $f$ is {\em chain transitive} if for any $x,y\in X$ and $\delta>0$, there is a $\delta$-chain $(x_i)_{i=0}^k$ of $f$ such that $x_0=x$ and $x_k=y$. 

The first result of this paper is the following theorem.

\begin{thm}
Given a homeomorphism $f\colon X\to X$, if there are $x,y\in X$ with $\mathcal{O}_f(x)\ne\mathcal{O}_f(y)$ and a closed subset $S$ of $X$ with $f(S)=S$ such that
\begin{itemize}
\item[(1)] $\{x,y\}\subset[W^u(S)\cap W^s(S)]\setminus S$,
\item[(2)] $\liminf_{i\to-\infty}d(f^i(x),f^i(y))=\liminf_{i\to+\infty}d(f^i(x),f^i(y))=0$,
\item[(3)] $f|_S\colon S\to S$ is chain transitive,
\item[(4)] $f$ has shadowing on $S\cup\mathcal{O}_f(x)\cup\mathcal{O}_f(y)$,
\end{itemize}
then $f$ is Bohr chaotic.
\end{thm}

Let us prove Theorem 1. For any $x\in X$ and $r>0$, we define subsets $B_r(x),C_r(x)$ of $X$ by
\[
B_r(x)=\{y\in X\colon d(x,y)\le r\},
\]
\[
C_r(x)=\{y\in X\colon d(x,y)\ge r\}.
\]

\begin{proof}[Proof of Theorem 1]
Let
\[
E=S\cup\mathcal{O}_f(x)\cup\mathcal{O}_f(y).
\]
By (1), we have $B_{3\epsilon}(x)\cap E=\{x\}$ and $B_{3\epsilon}(y)\cap E=\{y\}$ for some $\epsilon>0$.  By (4), there is $\delta>0$ such that every $2\delta$-pseudo orbit $(w_i)_{i\in\mathbb{Z}}$ of $f$ with $w_i\in E$ for all $i\in\mathbb{Z}$ is $\epsilon$-shadowed by some $p\in X$. By (1) and (2), there are $z,w\in S$ and $k,l\ge1$ such that
\[
\max\{d(w,f^{-k}(x)),d(w,f^{-k}(y)),d(f^l(x),z),d(f^l(y),z)\}\le\delta.
\]
By (3), there is a $\delta$-chain $(z_i)_{i=0}^j$ of $f|_S$ such that $z_0=z$ and $z_j=w$. Let $m=j+k+l$. We consider the following two $2\delta$-chains
\[
(x_i)_{i=0}^m=(z_0,z_1,\dots,z_{j-1},f^{-k}(x),\dots,f^{-1}(x),x,f(x),\dots,f^{l-1}(x),z)
\]
and
\[
(y_i)_{i=0}^m=(z_0,z_1,\dots,z_{j-1},f^{-k}(y),\dots,f^{-1}(y),y,f(y),\dots,f^{l-1}(y),z)
\]
of $f$. Given any bounded sequence $(a_i)_{i\ge0}$ of real numbers with
\[
\limsup_{n\to\infty}\frac{1}{n}\sum_{i=0}^{n-1}|a_i|>0,
\]
we have
\[
\limsup_{n\to\infty}\frac{1}{n}\sum_{h=1}^{n}|a_{r+hm}|>0
\]
for some $0\le r\le m-1$. We define a $2\delta$-pseudo orbit $\Gamma=(w_i)_{i\in\mathbb{Z}}$ of $f$ by
\begin{itemize}
\item $w_i=f^{i-r-l}(z)$ for all $i\le r+l$,
\item
\begin{equation*}
w_{r+l+(h-1)m+i}=
\begin{cases}
x_i&\text{if $a_{r+hm}>0$}\\
y_i&\text{if $a_{r+hm}\le0$}
\end{cases}
\end{equation*}
for all $h\ge1$ and $0\le i\le m-1$.
\end{itemize}
The $2\delta$-pseudo orbit $\Gamma=(w_i)_{i\in\mathbb{Z}}$ satisfies $w_i\in E$ for all $i\in\mathbb{Z}$ and so is $\epsilon$-shadowed by some $p\in X$. Note that
\begin{itemize}
\item
\begin{equation*}
w_{r+hm}=
\begin{cases}
x&\text{if $a_{r+hm}>0$}\\
y&\text{if $a_{r+hm}\le0$}
\end{cases}
\end{equation*}
for all $h\ge1$,
\item $w_i\in E\setminus\{x,y\}$ for all $i\ge0$ with $i\not\in\{r+hm\colon h\ge1\}$. 
\end{itemize}
We define $\phi\in C(X)$ by
\[
\phi(q)=\frac{d(q,C_{2\epsilon}(x))}{d(q,B_\epsilon(x))+d(q,C_{2\epsilon}(x))}-\frac{d(q,C_{2\epsilon}(y))}{d(q,B_\epsilon(y))+d(q,C_{2\epsilon}(y))}
\]
for all $q\in X$. Note that
\begin{equation*}
\phi(q)=
\begin{cases}
1&\text{for all $q\in B_\epsilon(x)$}\\
-1&\text{for all $q\in B_\epsilon(y)$}\\
0&\text{for all $r\in E\setminus\{x,y\}$ and $q\in B_\epsilon(r)$}
\end{cases}
.
\end{equation*}
It follows that
\[
\sum_{i=0}^{r+l+nm-1}a_i\phi(f^i(p))=\sum_{h=1}^n|a_{r+hm}|
\]
and so
\begin{align*}
\limsup_{n\to\infty}\frac{1}{n}\left|\sum_{i=0}^{n-1}a_i\phi(f^i(p))\right|&\ge\limsup_{n\to\infty}\frac{1}{r+l+nm}\left|\sum_{i=0}^{r+l+nm-1}a_i\phi(f^i(p))\right|\\
&\ge\frac{1}{m}\cdot\limsup_{n\to\infty}\frac{1}{n}\sum_{h=1}^{n}|a_{r+hm}|>0.
\end{align*}
Since $(a_i)_{i\ge0}$ is arbitrary, we conclude that $f$ is Bohr chaotic, completing the proof. 
\end{proof}

Given a homeomorphism $f\colon X\to X$ and $x\in X$, we define subsets $W^u(x),W^s(x)$ of $X$ by
\[
W^u(x)=\{y\in X\colon\lim_{i\to-\infty}d(f^i(x),f^i(y))=0\},
\]
\[
W^s(x)=\{y\in X\colon\lim_{i\to+\infty}d(f^i(x),f^i(y))=0\}.
\]
We obtain the following corollary of Theorem 1.

\begin{cor}
For a homeomorphism $f\colon X\to X$, if there are $x,y,z\in X$ with $\mathcal{O}_f(x)\ne\mathcal{O}_f(y)$ and a closed subset $S$ of $X$ with $f(S)=S$ such that
\begin{itemize}
\item[(1)] $z\in S$,
\item[(2)] $\{x,y\}\subset[W^u(z)\cap W^s(z)]\setminus S$,
\item[(3)] $f|_S\colon S\to S$ is chain transitive,
\item[(4)] $f$ has shadowing on $S\cup\mathcal{O}_f(x)\cup\mathcal{O}_f(y)$,
\end{itemize}
then $f$ is Bohr chaotic.
\end{cor}

A topological consequence of the shadowing lemma is that a $C^1$-diffeomorphism of a closed differentiable manifold is expansive and has shadowing on a neighborhood of its hyperbolic set (see, e.g., \cite[Theorem 18.1.2]{KH}). By using Corollary 1, we will obtain sufficient conditions for Bohr chaos in terms of shadowing and expansiveness (see Theorems 2 and 3 below). Those results directly apply to hyperbolic sets.

Our subsequent results are related to the notion of {\em chain components}. The definition is as follows. Let $f\colon X\to X$ be a homeomorphism. For any $x,y\in X$, the notation $x\rightarrow y$ means that for every $\delta>0$, there is a $\delta$-chain $(x_i)_{i=0}^k$ of $f$ with $x_0=x$ and $x_k=y$. The {\em chain recurrent set} $CR(f)$ for $f$ is defined by
\[
CR(f)=\{x\in X\colon x\rightarrow x\}.
\]
We define a relation $\leftrightarrow$ in
\[
CR(f)^2=CR(f)\times CR(f)
\]
by: for any $x,y\in CR(f)$, $x\leftrightarrow y$ if and only if $x\rightarrow y$ and $y\rightarrow x$. Note that $\leftrightarrow$ is a closed equivalence relation in $CR(f)^2$ and satisfies $x\leftrightarrow f(x)$ for all $x\in CR(f)$. An equivalence class $C$ of $\leftrightarrow$ is called a {\em chain component} for $f$. We denote by $\mathcal{C}(f)$ the set of chain components for $f$. The basic properties of chain components are the following
\begin{itemize}
\item $CR(f)=\bigsqcup_{C\in\mathcal{C}(f)}C$, a disjoint union,
\item $C$ is closed in $X$ and satisfies $f(C)=C$ for all $C\in\mathcal{C}(f)$,
\item $f|_C\colon C\to C$ is chain transitive for all $C\in\mathcal{C}(f)$.
\end{itemize}

\begin{rem}
\normalfont
For a homeomorphism $f\colon X\to X$ and $x\in X$, the {\em $\omega$-limit set} $\omega(x,f)$ (resp.\:{\em $\alpha$-limit set} $\alpha(x,f)$) of $x$ for $f$ is defined as the set of $y\in X$ such that $\lim_{j\to\infty}f^{i_j}(x)=y$ for some sequence $0\le i_1<i_2<\cdots$ (resp.\:$0\ge i_1>i_2>\cdots$). Since we have $y\rightarrow z$ for all $y,z\in\omega(x,f)$ (resp.\:$y,z\in\alpha(x,f)$) there is a unique $D\in\mathcal{C}(f)$ (resp.\:$C\in\mathcal{C}(f)$) such  that $\omega(x,f)\subset D$ (resp.\:$\alpha(x,f)\subset C$) and so 
\[
\lim_{i\to-\infty}d(f^i(x),C)=\lim_{i\to+\infty}d(f^i(x),D)=0.
\]
\end{rem}

{\em Expansiveness} is another feature of hyperbolic dynamics that has been extensively studied in the literature (see, e.g., \cite{AH}). We say that a homeomorphism $f\colon X\to X$ is {\em expansive} if there is $e>0$ such that
\[
\sup_{i\in\mathbb{Z}}d(f^i(x),f^i(y))\le e
\]
implies $x=y$ for all $x,y\in X$. In this case, $e>0$ is called an expansive constant for $f$. For a homeomorphism $f\colon X\to X$, we denote by $Per(f)$ the set of periodic points for $f$:
\[
Per(f)=\bigcup_{j\ge1}\{x\in X\colon f^j(x)=x\}.
\]
If $f$ is an expansive homeomorphism, then since
\[
\{x\in X\colon f^j(x)=x\}
\]
is a finite set for all $j\ge1$, $Per(f)$ is a countable set. If $e>0$ is an expansive constant for $f$, then for any $x,y\in X$, since
\[
E_n=\bigcap_{i\le-n}\{z\in X\colon d(f^i(x),f^i(z))\le e/2\}\cap\bigcap_{i\ge n}\{z\in X\colon d(f^i(y),f^i(z))\le e/2\}
\]
is a finite set for all $n\ge0$ and
\[
W^u(x)\cap W^s(y)\subset\bigcup_{n\ge0}E_n,
\]
$W^u(x)\cap W^s(y)$ is a countable set.

From the above observations, we obtain the following lemma.

\begin{lem}
For a homeomorphism $f\colon X\to X$, if
\begin{itemize}
\item $X$ is an uncountable set, 
\item $f$ is expansive,
\end{itemize}
then there is $C\in\mathcal{C}(f)$ such that $C$ is an infinite set.
\end{lem}

\begin{proof}
Assume the contrary, i.e., every $C\in\mathcal{C}(f)$ is a finite set. Then, every $C\in\mathcal{C}(f)$ is a periodic orbit. Given any $z\in X$, we have
\[
\lim_{i\to-\infty}d(f^i(z),C)=\lim_{i\to+\infty}d(f^i(z),D)=0
\]
for some $C,D\in\mathcal{C}(f)$. Since $C$ and $D$ are periodic orbits, there are $x\in C$ and $y\in D$ such that $z\in W^u(x)\cap W^s(y)$. Since $z\in X$ is arbitrary, we obtain
\[
X=\bigcup_{x,y\in Per(f)}\left[W^u(x)\cap W^s(y)\right].
\]
Since $f$ is expansive, it follows that $X$ is a countable set, a contradiction. This completes the proof of the lemma.
\end{proof}

Let $f\colon X\to X$ be a chain transitive homeomorphism. A {\em chain proximal relation} $\sim_f$ in
\[
X^2=X\times X
\]
is defined by: for any $x,y\in X$, $x\sim_f y$ if and only if for every $\delta>0$, there is a pair
\[
((x_i)_{i=0}^k,(y_i)_{i=0}^k)
\]
of $\delta$-chains of $f$ such that $(x_0,y_0)=(x,y)$ and $x_k=y_k$. It is known that $\sim_f$ is a closed $(f\times f)$-invariant equivalence relation in $X^2$. It is also known that for any $x,y\in X$, the following conditions are equivalent
\begin{itemize}
\item $x\sim_f y$,
\item for every $\delta>0$, there is $m\ge1$ such that for any $a,b\in\{x,y\}$, we have a $\delta$-chain $\gamma_{ab}=(z_i^{ab})_{i=0}^m$ of $f$ with $z_0^{ab}=a$ and $z_m^{ab}=b$.
\end{itemize}
If $x\sim_f y$ implies $x=y$ for all $x,y\in X$, then $X$ is a periodic orbit or an odometer (see the proof of Theorem 6 in \cite{RW}). Note that $X$ cannot be an odometer whenever $f$ is expansive. In other words, if
\begin{itemize}
\item $f$ is expansive,
\item $x\sim_f y$ implies $x=y$ for all $x,y\in X$,
\end{itemize}
then $X$ is a periodic orbit and so is a finite set.

\begin{rem}
\normalfont
The relation $\sim_f$ was introduced in \cite{S} and later rediscovered in \cite{RW} (based on the argument given in \cite[Exercise 8.22]{A}). 
\end{rem}

According to the above observation, we have the following lemma.

\begin{lem}
For a homeomorphism $f\colon X\to X$, if
\begin{itemize}
\item $X$ is an infinite set, 
\item $f$ is expansive and chain transitive,
\end{itemize}
then there are $x,y\in X$ such that $x\ne y$ and $x\sim_f y$.
\end{lem}

In order to state and prove the next theorem, we introduce a notation and a definition. For a closed subset $S$ of $X$ and $r>0$, we denote by $B_r(S)$ the closed $r$-neighborhood of $S$:
\[
B_r(S)=\{x\in X\colon d(x,S)\le r\}.
\] 
Given a homeomorphism $f\colon X\to X$ and a subset $S$ of $X$, we say that $f$ is {\em expansive} on $S$ if there is $e>0$ such that
\[
\sup_{i\in\mathbb{Z}}d(f^i(x),f^i(y))\le e
\]
implies $x=y$ for all $x,y\in S$ with $\mathcal{O}_f(x)\cup\mathcal{O}_f(y)\subset S$. In this case, $e>0$ is called an expansive constant for $f$ on $S$.

\begin{thm}
Given a homeomorphism $f\colon X\to X$, if there is a closed subset $C$ of $X$ with $f(C)=C$ such that
\begin{itemize}
\item $C$ is an infinite set,
\item $f|_C\colon C\to C$ is chain transitive,
\item $f$ is expansive and has shadowing on $B_b(C)$ for some $b>0$,
\end{itemize}
then $f$ is Bohr chaotic.
\end{thm}

By using Corollary 1 and Lemma 2, we prove Theorem 2.

\begin{proof}[Proof of Theorem 2]
Let $g=f|_C\colon C\to C$. Since
\begin{itemize}
\item $C$ is an infinite set,
\item $g$ is expansive and chain transitive,
\end{itemize}
by Lemma 2, there are $x,y\in C$ such that $x\ne y$ and $x\sim_g y$. Let $e>0$ be an expansive constant for $f$ on $B_b(C)$ and let
\[
0<\epsilon\le\min\{b,e/2,d(x,y)/3\}.
\]
Since $f$ has shadowing on $C$, there is $\delta>0$ such that every $\delta$-pseudo orbit $(z_i)_{i\in\mathbb{Z}}$ of $f$ with $z_i\in C$ for all $i\in\mathbb{Z}$ is $\epsilon$-shadowed by some $z\in X$. As $x\sim_g y$, there is $m\ge1$ such that for any $\alpha,\beta\in\{x,y\}$, we have a $\delta$-chain $\gamma_{\alpha\beta}=(w_i^{\alpha\beta})_{i=0}^m$ of $g$ with $w_0^{\alpha\beta}=\alpha$ and $w_m^{\alpha\beta}=\beta$. Given any $c=(c_j)_{j\in\mathbb{Z}}\in\{x,y\}^\mathbb{Z}$, we define a $\delta$-pseudo orbit $\Gamma_c=(z_i^c)_{i\in\mathbb{Z}}$ of $f$ by
\[
z_{jm+r}^c=w_r^{c_jc_{j+1}}
\]
for all $j\in\mathbb{Z}$ and $0\le r\le m-1$. Note that $z_i^c\in C$ for all $i\in\mathbb{Z}$. Let us consider a subset $Y$ of $X$ defined as
\[
Y=\{z\in X\colon\text{$\Gamma_c$ is $\epsilon$-shadowed by $z$ for some $c\in\{x,y\}^\mathbb{Z}$}\}.
\]
Note that $\mathcal{O}_f(z)\subset B_\epsilon(C)\subset B_b(C)$ for all $z\in Y$. We easily see that $Y$ is a closed subset of $X$ with $f^m(Y)=Y$. We define
a map
\[
\pi\colon Y\to\{x,y\}^\mathbb{Z}
\]
so that $\Gamma_{\pi(z)}$ is $\epsilon$-shadowed by $z$ for all $z\in Y$. We easily see that $\pi$ is a continuous surjection and satisfies
\[
\pi\circ f^m|_{Y}=\sigma\circ\pi,
\]
where
\[
\sigma\colon\{x,y\}^\mathbb{Z}\to\{x,y\}^\mathbb{Z}
\]
is the shift map. As $\epsilon\le e/2$, $\pi$ is injective and so is a homeomorphism. For every $j\in\{0\}\cup\mathbb{N}\cup\{\infty\}$, we define $p^{(j)}=(p^{(j)}_k)_{k\in\mathbb{Z}} (\in\{x,y\}^\mathbb{Z})$ by
\begin{equation*}
p^{(j)}_k=
\begin{cases}
y&\text{if $0\le k<j$}\\
x&\text{if $k\le -1$ or $j\le k$}
\end{cases}
\end{equation*}
for all $k\in\mathbb{Z}$. Let $q^{(j)}=\pi^{-1}(p^{(j)})$ ($\in Y$), $0\le j\le\infty$, and note that
\begin{itemize}
\item $q^{(j)}\ne q^{(j')}$ for all $0\le j<j'\le\infty$,
\item $f^m(q^{(0)})=q^{(0)}$,
\item $\lim_{L\to-\infty}f^{mL}(q^{(j)})=\lim_{L\to+\infty}f^{mL}(q^{(j)})=q^{(0)}$ for all $1\le j<\infty$,
\item $\lim_{j\to\infty}q^{(j)}=q^{(\infty)}$.
\end{itemize}
If $\mathcal{O}_f(q^{(1)})=\mathcal{O}_f(q^{(j)})$, i.e., $q^{(j)}\in\mathcal{O}_f(q^{(1)})$ for all $2\le j<\infty$, then we obtain $q^{(\infty)}\in\mathcal{O}_f(q^{(0)})$, which implies $f^m(q^{(\infty)})=q^{(\infty)}$ and so $\sigma(p^{(\infty)})=p^{(\infty)}$, a contradiction. It follows that  $\mathcal{O}_f(q^{(1)})\ne\mathcal{O}_f(q^{(j)})$ for some $2\le j<\infty$. Note that
\[
\mathcal{O}_f(q^{(0)})\cup\mathcal{O}_f(q^{(1)})\cup\mathcal{O}_f(q^{(j)})\subset B_b(C).
\]
Letting $S=\mathcal{O}_f(q^{(0)})$, we see that
\begin{itemize}
\item $\mathcal{O}_f(q^{(1)})\ne\mathcal{O}_f(q^{(j)})$,
\item $S$ is a closed subset of $X$ with $f(S)=S$,
\item $q^{(0)}\in S$,
\item $\{q^{(1)},q^{(j)}\}\subset[W^u(q^{(0)})\cap W^s(q^{(0)})]\setminus S$,
\item $f|_S\colon S\to S$ is chain transitive,
\item $f$ has shadowing on $S\cup\mathcal{O}_f(q^{(1)})\cup\mathcal{O}_f(q^{(j)})$.
\end{itemize}
By Corollary 1, we conclude that $f$ is Bohr chaotic, completing the proof of the theorem.
\end{proof}

\begin{rem}
\normalfont
The above proof shows that the assumptions of Theorem 2 ensure the existence of $m\ge1$ and a closed subset $Y$ of $X$ with $f^m(Y)\subset Y$ for which there is a homeomorphism $h\colon Y\to\{0,1\}^\mathbb{Z}$ such that
\[
h\circ f^m|_Y=\sigma\circ h,
\]
where $\sigma\colon\{0,1\}^\mathbb{Z}\to\{0,1\}^\mathbb{Z}$ is the shift map. As mentioned in Remark 1, by Theorem 1.1 of \cite{FFRS}, this condition implies that $f$ is Bohr chaotic. However, the proof of Theorem 1.1 in \cite{FFRS} relies on complicated arguments to construct a {\em horseshoe with disjoint steps}. By using Corollary 1, our proof of Theorem 2 avoids the arguments in \cite{FFRS}. Our proof of Theorem 1 is also independent of \cite[Theorem A]{HLT}.
\end{rem}

As a direct consequence of Lemma 1 and Theorem 2, we obtain the following theorem.

\begin{thm}
Given a homeomorphism $f\colon X\to X$, if there is a closed subset $\Lambda$ of $X$ with $f(\Lambda)=\Lambda$ such that
\begin{itemize}
\item $\Lambda$ is an uncountable set,
\item $f$ is expansive and has shadowing on $B_b(\Lambda)$ for some $b>0$,
\end{itemize}
then $f$ is Bohr chaotic.
\end{thm}

\begin{proof}
Since $\Lambda$ is an uncountable set and $f|_\Lambda\colon\Lambda\to\Lambda$ is expansive, by Lemma 1, there is $C\in\mathcal{C}(f|_\Lambda)$ such that $C$ is an infinite set. Since
\begin{itemize}
\item $f|_C\colon C\to C$ is chain transitive,
\item $f$ is expansive and has shadowing on $B_b(C)$,
\end{itemize}
Theorem 2 implies that $f$ is Bohr chaotic. 
\end{proof}

The following is a consequence of the shadowing lemma (see, e.g., \cite[Theorem 18.1.2]{KH}).

\begin{lem}
Let $M$ be a closed Riemannian manifold and let
\[
f\colon M\to M
\]
be a $C^1$-diffeomorphism. For any hyperbolic set $\Lambda$ for $f$, $f$ is expansive and has shadowing on $B_b(\Lambda)$ for some $b>0$.
\end{lem}

Finally, by combining Theorems 2 and 3 with Lemma 3, we obtain two corollaries on hyperbolic sets. 

\begin{cor}
Given a closed Riemannian manifold $M$ and a $C^1$-diffeomorphism 
\[
f\colon M\to M,
\]
if there is an infinite hyperbolic set $C$ for $f$ such that $f|_C\colon C\to C$ is chain transitive, then $f$ is Bohr chaotic.
\end{cor}

\begin{cor}
Given a closed Riemannian manifold $M$ and a $C^1$-diffeomorphism 
\[
f\colon M\to M,
\]
if there is an uncountable hyperbolic set $\Lambda$ for $f$, then $f$ is Bohr chaotic.
\end{cor}

\appendix

\section{}

Let $f\colon X\to X$ be a homeomorphism and let $\xi=(x_i)_{i\ge0}$ be a sequence of points in $X$. For $\delta>0$, $\xi$ is said to be a {\em $\delta$-pseudo orbit} of $f$ if
\[
\sup_{i\ge0}d(f(x_i),x_{i+1})\le\delta.
\]
For $x\in X$ and $\epsilon>0$, $\xi$ is said to be {\em $\epsilon$-shadowed} by $x$ if
\[
\sup_{i\ge0}d(x_i,f^i(x))\le\epsilon.
\]

The aim of this appendix is to give a direct proof of the following theorem.

\begin{theorem}
Let $\sigma\colon\{0,1\}^\mathbb{Z}\to\{0,1\}^\mathbb{Z}$ be the shift map. For a homeomorphism $f\colon X\to X$, if there are a closed  subset $Y$ of $X$, $m\ge1$, and a homeomorphism $h\colon Y\to\{0,1\}^\mathbb{Z}$ such that $f^m(Y)=Y$ and $h\circ f^m|_Y=\sigma\circ h$, then $f$ is Bohr chaotic.
\end{theorem}

\begin{proof}
As in the proof of Theorem 2, there are $p,z,w\in Y$ such that
\begin{itemize}
\item $p\ne z$, $z\ne w$, $w\ne p$,
\item $f^m(p)=p$,
\item $\mathcal{O}_f(z)\ne\mathcal{O}_f(w)$,
\item $\lim_{L\to-\infty}f^{mL}(z)=\lim_{L\to+\infty}f^{mL}(z)=p$,
\item $\lim_{L\to-\infty}f^{mL}(w)=\lim_{L\to+\infty}f^{mL}(w)=p$.
\end{itemize}
Let
\[
E=\mathcal{O}_f(p)\cup\mathcal{O}_f(z)\cup\mathcal{O}_f(w).
\]
Since $z,w$ are isolated points in $E$, we have $B_{3\epsilon}(z)\cap E=\{z\}$ and $B_{3\epsilon}(w)\cap E=\{w\}$ for some $\epsilon>0$. Since $\sigma$ has shadowing on $\{0,1\}^\mathbb{Z}$, $f^m|_Y$ has shadowing on $Y$. We take $\eta,\eta',\delta>0$ such that
\begin{itemize}
\item $d(u,v)\le\eta$ implies $\sup_{0\le i\le m-1}d(f^i(u),f^i(v))\le\epsilon/2$ for all $u,v\in X$,
\item $\eta'\le\epsilon/2$,
\item every $\eta'$-pseudo orbit of $f^m|_Y$ is $\eta$-shadowed by some point of $Y$,
\item every $\delta$-chain $(u_i)_{i=0}^m$ of $f$ satisfies $\sup_{0\le i\le m}d(u_i,f^i(u_0))\le\eta'$.
\end{itemize}
Let $(x_i)_{i\ge0}$ be a $\delta$-pseudo orbit of $f$ with $\{x_{\alpha m}\colon\alpha\ge0\}\subset Y$. We shall show that
\[
\sup_{i\ge0}d(x_i,f^i(x))\le\epsilon
\]
for some $x\in Y$. Let $\alpha\ge0$. Since $(x_{\alpha m+\beta})_{\beta=0}^m$ is a $\delta$-chain of $f$, we have
\[
\sup_{0\le\beta\le m}d(x_{\alpha m+\beta},f^\beta(x_{\alpha m}))\le\eta';
\]
therefore,
\[
d(x_{(\alpha+1)m},f^m(x_\alpha))\le\eta'.
\]
It follows that $(x_{\alpha m})_{\alpha\ge0}$ is an $\eta'$-pseudo orbit of $f^m|_{Y}$ and so $\eta$-shadowed by some $x\in Y$. For any $\alpha\ge0$, $0\le\beta\le m-1$, we obtain
\begin{align*}
d(x_{\alpha m+\beta},f^{\alpha m+\beta}(x))&\le d(x_{\alpha m+\beta},f^\beta(x_{\alpha m}))+d(f^\beta(x_{\alpha m}),f^\beta(f^{\alpha m}(x)))\\
&\le\eta'+\epsilon/2\le\epsilon/2+\epsilon/2=\epsilon.
\end{align*}
This implies $\sup_{i\ge0}d(x_i,f^i(x))\le\epsilon$, proving the claim. If $L\ge1$ is sufficiently large, then
\[
\gamma_z=(u_i)_{i=0}^{2mL}=(p,f^{-mL+1}(z),\dots,f^{-1}(z),z,f(z),\dots,f^{mL-1}(z),p)
\]
and
\[
\gamma_w=(v_i)_{i=0}^{2mL}=(p,f^{-mL+1}(w),\dots,f^{-1}(w),w,f(w),\dots,f^{mL-1}(w),p)
\]
are $\delta$-chains of $f$. Given $s=(s_n)_{n\ge1}\in\{z,w\}^\mathbb{N}$, let
\[
\xi^s=(x_i^s)_{i\ge0}=\gamma_{s_1}\gamma_{s_2}\gamma_{s_3}\cdots,
\]
a $\delta$-pseudo orbit of $f$. Since $\{x_{\alpha m}^s\colon\alpha\ge0\}\subset Y$, we have
\[
\sup_{i\ge0}d(x_i^s,f^i(x^s))\le\epsilon
\]
for some $x^s\in Y$. Note that
\begin{itemize}
\item $x_{(2n-1)mL}^s=s_n$ for all $n\ge1$,
\item $x_i^s\in E\setminus\{z,w\}$ if $i\ne(2n-1)mL$ for all $n\ge1$.
\end{itemize}
Given any bounded sequence $(a_i)_{i\ge0}$ of real numbers with
\[
\limsup_{n\to\infty}\frac{1}{n}\sum_{i=0}^{n-1}|a_i|>0,
\]
we have
\[
\limsup_{N\to\infty}\frac{1}{N}\sum_{n=1}^N|a_{j+(2n-1)mL}|>0
\]
for some $0\le j\le 2mL-1$. Let
\begin{equation*}
s_n=
\begin{cases}
z&\text{if $a_{j+(2n-1)mL}>0$}\\
w&\text{if $a_{j+(2n-1)mL}\le0$}
\end{cases}
\end{equation*}
for all $n\ge1$. We define $\phi\in C(X)$ by
\[
\phi(q)=\frac{d(q,C_{2\epsilon}(z))}{d(q,B_\epsilon(z))+d(q,C_{2\epsilon}(z))}-\frac{d(q,C_{2\epsilon}(w))}{d(q,B_\epsilon(w))+d(q,C_{2\epsilon}(w))}
\]
for all $q\in X$. Note that
\begin{equation*}
\phi(q)=
\begin{cases}
1&\text{for all $q\in B_\epsilon(z)$}\\
-1&\text{for all $q\in B_\epsilon(w)$}\\
0&\text{for all $r\in E\setminus\{z,w\}$ and $q\in B_\epsilon(r)$}
\end{cases}
.
\end{equation*}
It follows that
\[
\sum_{i=0}^{(2N-1)mL-1}a_{j+mL+i}\phi(f^{mL+i}(x^s))=\sum_{n=1}^N|a_{j+(2n-1)mL}|.
\]
Since $f$ is a homeomorphism, we have $x^s=f^j(y^s)$ for some $y^s\in X$. We obtain
\begin{align*}
\limsup_{n\to\infty}\frac{1}{n}\sum_{i=0}^{n-1}a_i\phi(f^i(y^s))&\ge\limsup_{N\to\infty}\frac{1}{j+2NmL}\sum_{\alpha=0}^{j+2NmL-1}a_\alpha\phi(f^\alpha(y^s))\\
&=\limsup_{N\to\infty}\frac{1}{j+2NmL}\sum_{\alpha=j+mL}^{j+2NmL-1}a_\alpha\phi(f^\alpha(y^s))\\
&=\limsup_{N\to\infty}\frac{1}{j+2NmL}\sum_{i=0}^{(2N-1)mL-1}a_{j+mL+i}\phi(f^{j+mL+i}(y^s))\\
&=\limsup_{N\to\infty}\frac{1}{j+2NmL}\sum_{i=0}^{(2N-1)mL-1}a_{j+mL+i}\phi(f^{mL+i}(x^s))\\
&=\limsup_{N\to\infty}\frac{1}{j+2NmL}\sum_{n=1}^N|a_{j+(2n-1)mL}|>0.\\
\end{align*}
Since $(a_i)_{i\ge0}$ is arbitrary, we conclude that $f$ is Bohr chaotic.
\end{proof}

\section{}

Let $\mathcal{H}(X)$ denote the set of homeomorphisms of $X$. We define a metric $D\colon\mathcal{H}(X)\times\mathcal{H}(X)\to[0,\infty)$ by
\[
D(f,g)=\max\left\{\max_{x\in X}d(f(x),g(x)),\max_{x\in X}d(f^{-1}(x),g^{-1}(x))\right\}
\]
for all $f,g\in\mathcal{H}(X)$. Note that $\mathcal{H}(X)$ is a complete metric space with respect to $D$.

The aim of this appendix is to prove the following theorem.

\begin{theorem}
Let $M$ be a closed Riemannian manifold. If $\dim M\ge2$, then generic $f\in\mathcal{H}(M)$ is Bohr chaotic.
\end{theorem}

This theorem is obtained in \cite{HLT}, but we give a different proof. For a homeomorphism $f\colon X\to X$, we denote by $h_{\rm top}(f)$ the topological entropy of $f$. Our proof of Theorem B.1 is through the following lemma.

\begin{lemma}
Given a homeomorphism $f\colon X\to X$, if
\begin{itemize}
\item $X$ is totally disconnected,
\item $f$ has shadowing on $X$ and satisfies $h_{\rm top}(f)>0$,
\end{itemize}
then $f$ is Bohr chaotic.
\end{lemma}

\begin{proof}[Proof of Lemma B.1]
This proof is based on the arguments in \cite{GM}. A {\em clopen partition} of $X$ is by definition a family of disjoint clopen subsets of $X$ whose union is $X$. For a clopen partition $\mathcal{P}$ of $X$, we define ${\rm mesh}(\mathcal{P})$ by
\[
{\rm mesh}(\mathcal{P})=\max_{A\in\mathcal{P}}\max_{a,b\in A}d(a,b).
\]
For two clopen partitions $\mathcal{P},\mathcal{Q}$ of $X$, the notation $\mathcal{P}\prec\mathcal{Q}$ means that for each $B\in\mathcal{Q}$, there is $A\in\mathcal{P}$ such that $B\subset A$. Since $X$ is totally disconnected, we have a sequence $(\mathcal{P}_n)_{n\ge1}$ of clopen partitions of $X$ such that
\begin{itemize}
\item $\mathcal{P}_n\prec\mathcal{P}_{n+1}$ for all $n\ge1$,
\item $\lim_{n\to\infty}{\rm mesh}(\mathcal{P}_n)=0$.
\end{itemize}
Let $\sigma_n\colon\mathcal{P}_n^\mathbb{Z}\to\mathcal{P}_n^\mathbb{Z}$ denote the shift map for all $n\ge1$. Given $n\ge1$, let
\[
X^{(n)}=\left\{x^{(n)}=(x_i^{(n)})_{i\in\mathbb{Z}}\in\mathcal{P}_n^\mathbb{Z}\colon\text{$x_i^{(n)}\cap f^{-1}(x_{i+1}^{(n)})\ne\emptyset$ for all $i\in\mathbb{Z}$}\right\}
\]
and note that $X^{(n)}$ is a subshift of finite type. Let
\[
f_n=\sigma_n|_{X^{(n)}}\colon X^{(n)}\to X^{(n)}.
\]
For every $n\ge1$, we define a map $\pi_n\colon X^{(n+1)}\to X^{(n)}$ by: for any $x^{(n)}\in X^{(n)},x^{(n+1)}\in X^{(n+1)}$, $\pi_n(x^{(n+1)})=x^{(n)}$ if and only if $x_i^{(n+1)}\subset x_i^{(n)}$ for all $i\in\mathbb{Z}$. Note that
\[
\pi_n\circ f_{n+1}=f_n\circ\pi_n.
\]
Let
\[
Y=\left\{(x^{(n)})_{n\ge1}\in\prod_{n\ge1}X^{(n)}\colon\text{$\pi_n(x^{(n+1)})=x^{(n)}$ for all $n\ge1$}\right\}.
\]
We define a map $g\colon Y\to Y$ by
\[
g((x^{(n)})_{n\ge1})=(f_n(x^{(n)}))_{n\ge1}
\]
for all $(x^{(n)})_{n\ge1}\in Y$. We also define a map $h\colon Y\to X$ by
\[
h((x^{(n)})_{n\ge1})\in\bigcap_{n\ge1}x_0^{(n)}
\] 
for all $(x^{(n)})_{n\ge1}\in Y$. Note that $h$ is a homeomorphism and satisfies $h\circ g=f\circ h$. For any $m\ge n\ge1$, we define a map $\pi_n^m\colon X^{(m)}\to X^{(n)}$ by
\begin{equation*}
\pi_n^m=
\begin{cases}
id_{X^{(n)}}&\text{if $m=n$}\\
\pi_n\circ\pi_{n+1}\circ\cdots\circ\pi_{m-1}&\text{if $m\ge n+1$}
\end{cases}
.
\end{equation*}
Since $f$ has shadowing on $X$,
\[
\pi=(\pi_n\colon X^{(n+1)}\to X^{(n)})_{n\ge1}
\]
satisfies the Mittag--Leffler condition, i.e., for any $N\ge1$, there is $M\ge N$ such that
\[
\pi_N^M(X^{(M)})=\pi_N^m(X^{(m)})
\]
for all $m\ge M$. Given $N\ge1$, we define a map $\pi^{(N)}\colon Y\to X^{(N)}$ by $\pi^{(N)}((x^{(n)})_{n\ge1})=x^{(N)}$ for all $(x^{(n)})_{n\ge1}\in Y$. We have $\pi^{(N)}\circ g=f_N\circ\pi^{(N)}$. Let $X_N=\pi^{(N)}(Y)$ and note that $f_N(X_N)=X_N$. Since
\[
h_{\rm top}(g)=h_{\rm top}(f)>0,
\]
we have $h_{\rm top}(f_N|_{X_N})>0$ for some $N\ge1$. By the Mittag--Leffler condition, if $M\ge N$ is sufficiently large, then $X_N=\pi_N^M(X^{(M)})$. Since $X^{(M)}$ is a subshift of finite type, $X_N$ is a {\em sofic shift}. By \cite[Proposition 3]{M}, we obtain a subshift of finite type $\Sigma\subset X_N$ such that $h_{\rm top}(f_N|_{\Sigma})>0$. Since $\Sigma$ is a subshift of finite type, $f_N|_{\Sigma}$ is expansive and has shadowing on $\Sigma$. Since $h_{\rm top}(f_N|_{\Sigma})>0$, $\Sigma$ is an uncountable set. By Theorem 3, $f_N|_{\Sigma}$ is Bohr chaotic. It follows that $f_N|_{X_N}$ is Bohr chaotic and so is $g$; therefore, we conclude that $f$ is Bohr chaotic, completing the proof.
\end{proof}

Let $M$ be a closed Riemannian manifold such that $\dim M\ge2$. Let
\begin{itemize}
\item $\mathcal{H}_{zdcr}(M)=\{f\in\mathcal{H}(M)\colon\text{$CR(f)$ is totally disconnected}\}$,
\item $\mathcal{H}_{sh}(M)=\{f\in\mathcal{H}(M)\colon\text{$f$ has shadowing on $M$}\}$,
\item $\mathcal{H}_{pte}(M)=\{f\in\mathcal{H}(M)\colon\text{$h_{\rm top}(f)>0$}\}$.
\end{itemize}
By \cite{AHK,PP,Y}, we know that $\mathcal{H}_w(M)$, $w\in\{zdcr,sh,pte\}$, are residual subsets of $\mathcal{H}(M)$, thus
\[
\mathcal{H}_{zdcr}(M)\cap\mathcal{H}_{sh}(M)\cap\mathcal{H}_{pte}(M)
\]
is a residual subset of $\mathcal{H}(M)$. Given any $f\in\mathcal{H}_{zdcr}(M)\cap\mathcal{H}_{sh}(M)\cap\mathcal{H}_{pte}(M)$,\begin{itemize}
\item $CR(f)$ is totally disconnected,
\item $f|_{CR(f)}$ has shadowing on $CR(f)$,
\item $h_{\rm top}(f|_{CR(f)})=h_{\rm top}(f)>0$;
\end{itemize}
therefore, by Lemma B.1, $f|_{CR(f)}$ is Bohr chaotic and so is $f$. This completes the proof of Theorem B.1.

\end{document}